\documentclass[letterpaper,11pt]{article}

\usepackage{times}

\usepackage{amsthm,amsmath,amsfonts}

\usepackage{graphicx}
\usepackage[small]{caption}
\usepackage[hidelinks]{hyperref}
\newcommand{\orcid}[1]{\,\href{https://orcid.org/#1}{\includegraphics[width=8pt]{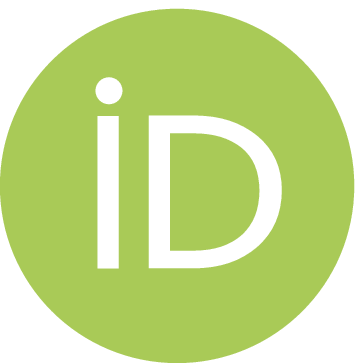}}}

\usepackage{color}

\newcommand{\A}{\mathcal{A}}
\newcommand{\B}{\mathcal{B}}
\newcommand{\C}{\mathcal{C}}
\newcommand{\D}{\mathcal{D}}
\newcommand{\E}{\mathcal{E}}
\newcommand{\F}{\mathcal{F}}
\newcommand{\I}{\mathcal{I}}
\newcommand{\J}{\mathcal{J}}
\newcommand{\K}{\mathcal{K}}
\renewcommand{\P}{\mathcal{P}}
\newcommand{\Q}{\mathcal{Q}}
\newcommand{\R}{\mathcal{R}}
\renewcommand{\S}{\mathcal{S}}
\newcommand{\U}{\mathcal{U}}

\newtheorem{theorem}{Theorem}
\newtheorem{lemma}{Lemma}
\newtheorem{proposition}{Proposition}

\title{Vertebrate interval graphs}

\author{Rain Jiang\orcid{0000-0002-0144-942X}\qquad
Kai Jiang\orcid{0000-0001-8165-0571}\qquad
Minghui Jiang\orcid{0000-0003-1843-9292}\,\thanks{\texttt{ dr.minghui.jiang at gmail.com}}\medskip\\
Home School, USA}

\date{}

\begin{document}

\maketitle

\begin{abstract}
A vertebrate interval graph is an interval graph in which
the maximum size of a set of independent vertices equals the number of maximal cliques.
For any fixed $v \ge 1$,
there is a polynomial-time algorithm for deciding
whether a vertebrate interval graph admits a vertex partition into
two induced subgraphs with claw number at most $v$.
In particular, when $v = 2$,
whether a vertebrate interval graph can be partitioned
into two proper interval graphs
can be decided in polynomial time.
\end{abstract}

\section{Introduction}

In any graph, the maximum size of a set of independent vertices is at most
the number of maximal cliques.
It is natural to consider graphs with the property that these two numbers are equal.
Golumbic~\cite{Go78} noted that
``we cannot expect to discover much about
the structure'' of graphs with this property alone,
and instead defined \emph{trivially perfect graphs}
as graphs with the stronger, hereditary property that
the maximum size of a set of independent vertices
equals the number of maximal cliques
in every induced subgraph.

Trivially perfect graphs are also known as comparability graphs of trees,
and as intersection graphs of nested intervals.
An \emph{interval graph} is the intersection graph of a family of open
intervals.
A trivially perfect graph is an interval graph
with an interval representation in which no interval partially overlaps
another interval.

\begin{figure}[htbp]
\centering\includegraphics{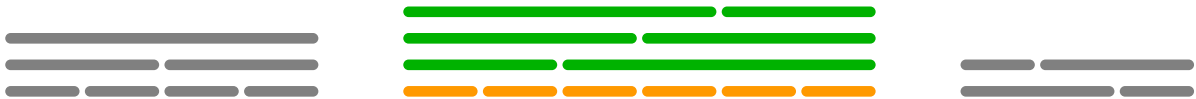}
\caption{Left: intervals representing a trivially perfect graph.
Middle: intervals representing a vertebrate interval graph.
Right: intervals representing an invertebrate interval graph.}
\label{fig:example}
\end{figure}

In this paper, we introduce \emph{vertebrate interval graphs}
as a subclass of interval graphs
and a superclass of trivially perfect graphs.
Specifically,
an interval graph is
\emph{vertebrate}
if the maximum size of a set of independent vertices equals the number of maximal cliques
in the graph,
and is \emph{invertebrate} otherwise.
Here the equality of the two numbers for a vertebrate interval graph is not hereditary:
it holds for the graph as a whole,
but not necessarily for its induced subgraphs.
Refer to Figure~\ref{fig:example} for some examples.

The \emph{claw number} $\psi(G)$ of a graph $G$ is the largest number
$v \ge 0$ such that $G$ contains $K_{1,v}$ as an induced subgraph~\cite{AC10}.
In particular, a graph with claw number $0$ is an empty graph with no edges.
A graph with claw number at most $1$ is a disjoint union of cliques,
also known as a \emph{cluster graph}.
Every cluster graph is clearly also an interval graph,
with claw number at most $1$.
A \emph{proper interval graph} is the intersection graph of a family of open
intervals in which no one properly contains another.
A \emph{unit interval graph} is the intersection graph of a family of open
intervals of the same length.
It is well-known~\cite{BW99}
that an interval graph is a proper interval graph if and only if it is a unit
interval graph, and if and only if it does not contain $K_{1,3}$ as an induced subgraph.
Thus a proper\,/\,unit interval graph is an interval graph of claw number at most~$2$.

A \emph{$k$-partition} of a graph refers to a partition of its vertices into $k$ subsets,
and the resulting $k$ vertex-disjoint induced subgraphs.
For $k \ge 2$ and $v \ge 1$,
\textsc{$k$-Partition$(v)$} is the problem of
deciding whether a given graph admits a vertex partition
into $k$ induced subgraphs with claw number at most $v$.
Certain generic results on vertex-partitioning~\cite{Ac97,Fa04}
imply that
\textsc{$k$-Partition$(v)$} in general graphs
is NP-hard for all $k \ge 2$ and $v \ge 1$.

The subchromatic number of a graph
is the smallest number $k$ such that
the graph admits a $k$-partition into cluster graphs~\cite{AJHL89}.
For any fixed $k$,
whether an interval graph with $n$ vertices has subchromatic number at most $k$
can be decided in $O(k\cdot n^{2 k+1})$ time~\cite[Theorem~5.4]{BFNW02}.
Thus for $k \ge 2$,
\textsc{$k$-Partition$(1)$} in interval graphs
admits an algorithm running in $O(k\cdot n^{2 k+1})$ time.
It is unknown~\cite{JJJ21}
whether there exist $k \ge 2$ and $v \ge 2$ such that
\textsc{$k$-Partition$(v)$} in interval graphs
admits an polynomial-time algorithm.

In this paper,
we show that vertebrate interval graphs have sufficient ``structure'' so that
\textsc{$2$-Partition$(v)$} can be solved in polynomial time for all $v \ge 1$:

\begin{theorem}\label{thm:2partition}
For $v \ge 1$,
\textsc{$2$-Partition$(v)$} in vertebrate interval graphs
admits an algorithm running in $2^{O(v^2)}n^{O(v)}$ time.
\end{theorem}

In particular, for $v = 2$,
whether a vertebrate interval graph
admits a vertex partition into two proper interval graphs
can be decided in polynomial time.

\section{Preliminaries}

Denote by $m(G)$ the number of maximal cliques in a graph $G$.
Besides $m(G)$ and the claw number $\psi(G)$ mentioned earlier,
there are two other common graph parameters.
An \emph{independent set} (respectively, a \emph{clique})
is a set of pairwise non-adjacent (respectively, adjacent)
vertices in a graph.
The \emph{independence number}
$\alpha(G)$ is the maximum number of vertices in an independent set
in $G$.
The \emph{vertex-clique-partition number}
$\vartheta(G)$ is the minimum number of parts
in a vertex partition of $G$ into cliques.

It is easy to see that for any graph $G$,
$\psi(G) \le \alpha(G) \le \vartheta(G) \le m(G)$.
If $G$ is an interval graph,
then
$\alpha(G) = \vartheta(G)$
\cite{Go04}.
By our definition,
$\alpha(G) = \vartheta(G) = m(G)$
for any vertebrate interval graph $G$.

The equality
$\alpha(G) = \vartheta(G)$ for every interval graph $G$
has a simple constructive proof.
Recall that interval graphs can be recognized in linear time~\cite{Go04}.
Let $\I$ be any family of open intervals whose intersection graph is $G$.
Assume without loss of generality that all endpoints of intervals in $\I$ are integers.
We can find an independent set and a vertex clique partition at the same time
by a standard \emph{sweepline algorithm}:
\begin{quote}
Initialize $\I' \gets \I$ and $i \gets 1$.
While $\I'$ is not empty,
let $T_i = (l_i, r_i)$ be an interval in $\I'$ whose right endpoint $r_i$ is
minimum,
let $\I'_i$ (respectively, $\I_i$) be the subfamily of intervals in
$\I'$ (respectively, $\I$) that contain
the subinterval $S_i = (r_i - 1, r_i)$ of $T_i$,
then update $\I' \gets \I' \setminus \I'_i$ and $i \gets i + 1$.
\end{quote}

Let $m$ be the maximum round $i$ such that $\I'$ is not empty.
Then
$S_i \subseteq T_i \in \I'_i \subseteq \I_i$
for $1 \le i \le m$.
Note that the $m$ intervals $T_i$ correspond to an independent set in $G$,
and that the $m$ subfamilies $\I'_i$ correspond to a partition of the vertices of $G$
into cliques.
So we have $\alpha(G) = \vartheta(G) = m$.
Moreover,
if $G$ is a vertebrate interval graph,
then $\alpha(G) = \vartheta(G) = m(G) = m$,
and the $m$ subfamilies $\I_i$
correspond to the $m$ maximal cliques in $G$,
which cover all vertices and all edges.

After running the sweepline algorithm on
an arbitrary interval representation $\I$ of a vertebrate interval graph $G$,
we can then obtain a possibly more compact interval representation $\J$ of $G$,
by converting each interval $I \in \I$ to an interval $J \in \J$ as follows.
Suppose that $a$ and $b$, respectively,
are the smallest and largest integers $i$, $1 \le i \le m$,
such that $I \in \I_i$.
Then $I \in \I_i$ if and only if $a \le i \le b$.
Let $J = (a - 1, b)$.

Note that all intervals in $\J$ have integer endpoints in $[0,m]$.
Moreover,
for each interval $I \in \I$,
the length of the corresponding interval $J \in \J$
is equal to the number of maximal cliques $\I_i$ such that $I \in \I_i$.
In particular,
each interval $T_i \in \I$ corresponds to a unit interval $T'_i = (i-1,i)$ in $\J$,
and each interval $J \in \J$ of length $\ell \ge 2$ is the concatenation of $\ell$
disjoint unit intervals $T'_i$ in $\J$.
It is easy to verify that $\I$ and $\J$ represent the same graph $G$.

Let $\U$ be the subfamily of $m$ unit intervals $T'_i$ in $\J$.
We call $\J$ a \emph{vertebrate representation} of $G$,
and call $\U$ the \emph{backbone} of $\J$.
Because of the unit intervals composing the backbone,
the maximum length of an interval in $\J$ is exactly $\psi(G)$.
In summary, we have the following proposition:

\begin{proposition}
Every vertebrate interval graph $G$ has an interval representation $\J$
with integer endpoints in $[0,m(G)]$ and with maximum interval length $\psi(G)$,
which includes the $m(G)$ unit intervals $(i-1,i)$ for $1 \le i \le m(G)$.
\end{proposition}

\section{Algorithm for \textsc{$2$-Partition$(v)$} in vertebrate interval graphs}

In this section we prove Theorem~\ref{thm:2partition}.
Fix any $v \ge 1$.
We will present a polynomial-time algorithm for \textsc{$2$-Partition$(v)$}
based on dynamic programming.

Let $G$ be a vertebrate interval graph with $n$ vertices.
Obtain a vertebrate representation $\J$ of $G$
with integer endpoints in $[0,m]$,
where $m = m(G)$ is the number of maximal cliques in $G$.
Let $\U \subseteq \J$ be the subfamily of $m$ disjoint unit intervals
$(i-1,i)$, $1 \le i \le m$, forming the backbone.

We say that a family of intervals is \emph{good}
if it represents an interval graph with claw number at most $v$,
and that a $2$-partition of a family of intervals is \emph{good}
if each part is good.
For any two subfamilies $\R$ and $\S$ of $\J$,
we write
$\R\mid\S$ if each interval in $\R$
intersects at most $v$ disjoint intervals in $\S$ other than itself.
Then a $2$-partition $(\P,\Q)$ of $\J$ is good if and only if
$\P\mid\P$ and $\Q\mid\Q$.

For any good $2$-partition of $\J$,
if there are two duplicate intervals in different parts,
then we can always relocate them to the same part
while keeping the $2$-partition good.
Thus we can assume without loss of generality
that $\J$ includes no duplicates.

\subsection{Basic properties of good $2$-partitions}

For any subfamily $\S$ of $\J$,
denote by $\S_\le$ the subfamily of intervals in $\S$ of length at most $v$,
and by $\S_>$ the subfamily of intervals in $\S$ of length greater than $v$.
For $0 \le l \le r \le m$,
denote by $\S(l,r)$ the subfamily of intervals in $\S$
that are contained in the interval $(l,r)$.
In particular, when $l = r$, both $(l,r)$ and $\S(l,r)$ are empty.
We write $(\P,\Q) \succ \S$
if $(\P,\Q)$ is a $2$-partition of $\S$.
We say that $(\P,\Q) \succ \S$ is \emph{simple}
if either $(\P,\Q) = (\S_\le,\S_>)$ or
$(\P,\Q) = (\S_>,\S_\le)$.

For any two subfamilies $\R$ and $\S$ of $\J$,
we say that two $2$-partitions
$(\R_1,\R_2) \succ \R$
and
$(\S_1,\S_2) \succ \S$
are \emph{consistent}
if $(\R_1 \cup \S_1, \R_2 \cup \S_2) \succ \R \cup \S$,
and write
$(\R_1,\R_2) \simeq (\S_1,\S_2)$.
For $\R \subseteq \S$,
and for any $2$-partition $(\P,\Q) \succ \S$,
the $2$-partition $(\P\cap\R,\Q\cap\R) \succ \R$
is called the \emph{derived} $2$-partition of $\R$.
Clearly, if $(\P',\Q')$ is derived from $(\P,\Q)$,
then $(\P',\Q') \simeq (\P,\Q)$.

For $0 \le l < r \le s \le m$,
the interval $(l,r)$ is a \emph{vertebrate range}
for some $2$-partition of $\J(0,s)$
if the $r-l$ unit intervals in $\U(l,r)$
are all in the same part.
A $2$-partition of $\J(0,s)$ is \emph{basic} if
for every maximal vertebrate range $(l,r)$,
the derived $2$-partition of $\J(l,r)$ is simple.

\begin{lemma}\label{lem:basic}
$\J$ has a good $2$-partition if and only if
it has a good basic $2$-partition.
\end{lemma}

\begin{proof}
Given any good $2$-partition of $\J$ that is not basic,
we can transform it into a good basic $2$-partition as follows.
Let $(l,r)$ be a maximal vertebrate range,
such that the $r-l$ unit intervals in $\J(l,r)$ are in the same part.
Then the intervals of length greater than $v$ in $\J(l,r)$,
if any, must be in the other part.
For each interval $I$ of length at most $v$ in $\J(l,r)$,
we simply move it to the same part as the unit intervals
if it is not already there.
We claim that this move does not introduce any star $K_{1,v+1}$.
First,
$I$ cannot be the center of any star $K_{1,v+1}$,
since any interval of length at most $v$ intersects at most $v$ disjoint
intervals in $\J$.
Second, if there is a star $K_{1,v+1}$ with $I$ as a leaf after the move,
then there would have been a star $K_{1,v+1}$ even before the move,
with $I$ replaced by some unit interval contained in $I$.
\end{proof}

For any subfamily $\S$ of $\J$,
and for $0 \le l \le r \le m$,
denote by $\alpha(\S,l,r)$
the maximum size of a subfamily of disjoint intervals in $\S$
that intersect the interval $(l,r)$.
In particular, when $l = r$, $(l,r)$ is empty,
and hence $\alpha(\S,l,r) = 0$.
Note that
$\alpha(\S,l,r)$
is not necessarily equal to
$\alpha(\S(l,r),l,r)$.

\begin{lemma}\label{lem:cut}
For $0 \le a < s < b \le m$,
if $(\P,\Q)$ is a $2$-partition of $\J$ with
$(s-1,s)\in\P$ and $(s,s+1)\in\Q$,
then $\alpha(\P,a,b) = \alpha(\P(0,s),a,s) + \alpha(\P(s,m),s,b)$
and $\alpha(\Q,a,b) = \alpha(\Q(0,s),a,s) + \alpha(\Q(s,m),s,b)$.
\end{lemma}

\begin{proof}
Consider any subfamily of disjoint intervals in $\J$ that intersect $(a,b)$,
from the same part $\P$ or $\Q$.
If this subfamily contains an interval $(c,d)$ with $c < s < d$,
then we can replace it by $(s-1,s)$ or $(s,s+1)$.
\end{proof}

For $0 \le s \le m$,
a sequence
$\boldsymbol{r} = \langle r_0, r_1, \ldots, r_v, r_{v+1}, r_{v+2} \rangle$
of $v+3$ integers is \emph{$s$-monotonic} if
$$
-1 = r_{v+2} \le r_{v+1} \le r_v \le \ldots \le r_1 \le r_0 = s,
$$
and moreover, either $-1 = r_{u+1} = r_u$ or $r_{u+1} < r_u$ holds for each $u$,
$0 \le u \le v + 1$.
Denote by $\boldsymbol{0}$ the unique $0$-monotonic sequence
$\langle 0, -1,\ldots, -1 \rangle$.

For $0 \le u \le v + 2$
and $0 \le s \le m$,
and for a subfamily $\R$ of $\J(0,s)$,
let $i(\R,u,s)$ be the maximum index $i\in[0,s]$ such that $u \le \alpha(\R,i,s) \le v+1$,
or $-1$ if such an index $i$ does not exist.
Clearly, we always have $i(\R,0,s) = s$ and $i(\R,v+2,s) = -1$.
Also, $i(\R,u,s) = -1$ for $u > \alpha(\R,0,s)$.
For $1 \le u \le \min\{\alpha(\R,0,s), v+1\}$,
we have $i(\R,u,s)\in[0,s]$ and $\alpha(\R,i(\R,u,s),s) = u$.

For $0 \le s \le m$,
we say that
an $s$-monotonic sequence $\boldsymbol{r}$
\emph{encodes}
a subfamily $\R$ of $\J(0,s)$
if
$r_u = i(\R,u,s)$ for $0 \le u \le v+2$.
By definition,
every subfamily of $\J(0,s)$ is encoded by
an $s$-monotonic sequence.
If $\boldsymbol{r}$ encodes $\R$,
that is, if $r_u = i(\R,u,s)$ for $0 \le u \le v+2$,
then from $\boldsymbol{r}$ we can recover the values of $\alpha(\R,i,s)$
relative to $v+1$
for all $0 \le i \le s$ as follows:
$$
\begin{array}{ll}
\alpha(\R,i,s) = u
&\textrm{when } r_{u+1} < i \le r_u \textrm{ for } 0 \le u \le v,\\
\alpha(\R,i,s) \ge v+1
&\textrm{when } r_{v+2} < i \le r_{v+1}.
\end{array}
$$

For $0 \le i \le s \le m$,
and for an $s$-monotonic sequence $\boldsymbol{r}$,
denote by
$\alpha(\boldsymbol{r}, i, s)$
the maximum index $u \in [0,v+1]$ such that
$i \le r_u$.
Then we immediately have the following lemma:

\begin{lemma}\label{lem:alpha}
For $0 \le s \le m$,
if an $s$-monotonic sequence $\boldsymbol{r}$ encodes a subfamily $\R$ of $\J(0,s)$,
then for $0 \le i \le s$,
either $\alpha(\R,i,s) = \alpha(\boldsymbol{r},i,s) \in[0,v]$,
or $\alpha(\R,i,s) \ge \alpha(\boldsymbol{r},i,s) = v+1$.
\end{lemma}

For $0 \le s \le m$,
we say that a pair of $s$-monotonic sequences $(\boldsymbol{p},\boldsymbol{q})$
\emph{encodes} a $2$-partition $(\P,\Q) \succ \J(0,s)$
if
$\boldsymbol{p}$ encodes $\P$
and
$\boldsymbol{q}$ encodes $\Q$.

For $0 \le s \le m$,
denote by $\K_s$ the subfamily of intervals $(a, b)$ in $\J$ with $a < s < b$.
Note that the $m$ unit intervals in $\U$ are the only intervals in $\J$
that are not included in any $\K_s$.
For $0 \le s' < s \le m$,
$\J(0,s) = \J(0,s')\cup(\K_{s'}\setminus\K_{s})\cup\J(s',s)$.
We say that
a pair of $s$-monotonic sequences $(\boldsymbol{p},\boldsymbol{q})$
\emph{extends}
a pair of $s'$-monotonic sequences $(\boldsymbol{p'},\boldsymbol{q'})$
by a $2$-partition
$(\E,\F) \succ \K_{s'}\setminus\K_{s}$ if,
with the simple $2$-partition $(\C,\D) = (\J_\le(s',s),\J_>(s',s)) \succ \J(s',s)$,
and with
$w = \alpha(\D,s',s)$
and
$w' = \alpha(\F\cup\D,s',s)$,
\begin{align*}
p_u &= \left\{
\begin{array}{ll}
s-u
& \quad\textrm{for } 1 \le u \le \min\{s-s',v+1\},\\
p'_{u-s+s'}
& \quad\textrm{for } s-s' < u \le v+1,
\end{array}
\right.
\\
q_u
&= \left\{
\begin{array}{ll}
i(\F\cup\D,u,s)
& \quad\textrm{for } 1 \le u \le \min\{w',v+1\},\\
q'_{u-w}
& \quad\textrm{for } w' < u \le v+1.
\end{array}
\right.
\end{align*}
Note that
if $(\boldsymbol{p},\boldsymbol{q})$
extends
$(\boldsymbol{p'},\boldsymbol{q'})$
by $(\E,\F)$,
then $(\boldsymbol{p},\boldsymbol{q})$
is uniquely determined by
$(\boldsymbol{p'},\boldsymbol{q'})$
and $(\E,\F)$.

The following lemma shows that basic $2$-partitions
have a sequential structure with respect to maximal vertebrate ranges:

\begin{lemma}\label{lem:ef}
Let $0 \le s' < s \le m$.
Suppose that
\begin{itemize}\setlength\itemsep{0pt}

\item
$(\P,\Q)$ is a basic $2$-partition of $\J(0,s)$
with $(s',s)$ as the rightmost maximal vertebrate range,

\item
$(\C,\D) = (\J_\le(s',s),\J_>(s',s))$
is the simple $2$-partition of $\J(s',s)$ derived from $(\P,\Q)$,

\item
$(\P',\Q')$ is the basic $2$-partition of $\J(0,s')$ derived from $(\P,\Q)$,

\item
$(\E,\F) = (\P\setminus(\P'\cup\C),\Q\setminus(\Q'\cup\D))$
is the $2$-partition of $\K_{s'}\setminus\K_s$ derived from $(\P,\Q)$,

\item
$(\boldsymbol{p},\boldsymbol{q})$
is a pair of $s$-monotonic sequences
encoding
$(\P,\Q)$,

\item
$(\boldsymbol{p'},\boldsymbol{q'})$
is a pair of $s'$-monotonic sequences
encoding
$(\P',\Q')$.

\end{itemize}
Then
$(\boldsymbol{p},\boldsymbol{q})$
extends
$(\boldsymbol{p'},\boldsymbol{q'})$
by $(\E,\F)$.
\end{lemma}

\begin{proof}
Since $\U(s',s) \subseteq \J_\le(s',s) = \C \subseteq \P$,
we must have
$\alpha(\P,s',s) = s-s'$,
and moreover,
$$
\begin{array}{ll}
i(\P,u,s) = s-u
&\textrm{for } 1 \le u \le \min\{s-s',v+1\},\\
i(\P,u,s) < s'
&\textrm{for } s-s' < u \le v+1.
\end{array}
$$

Note that all intervals in
$\Q = \Q'\cup\F\cup\D$
that intersect $(s',s)$ are in
$\F\cup\D$.
Thus
$\alpha(\Q,i,s) = \alpha(\F\cup\D,i,s)$
for $s' \le i \le s$.
Since $\alpha(\F\cup\D,s',s) = w'$,
it follows that
for $u \le \min\{w',v+1\}$,
$i(\F\cup\D,u,s) \ge s'$,
and hence
$i(\Q,u,s) = i(\F\cup\D,u,s)$.
Thus,
$$
\begin{array}{ll}
i(\Q,u,s) = i(\F\cup\D,u,s)
&\textrm{for } 1 \le u \le \min\{w',v+1\},\\
i(\Q,u,s) < s'
&\textrm{for } w' < u \le v+1.
\end{array}
$$

We next elaborate on the two inequalities,
$i(\P,u,s) < s'$
for $s-s' < u \le v+1$,
and
$i(\Q,u,s) < s'$
for $w' < u \le v+1$,
by considering two cases.

Case 1:
$s'=0$.
Then we must have
$i(\P,u,s) = -1$ for $s-s' < u \le v+1$,
and $i(\Q,u,s) = -1$ for $w' < u \le v+1$.
On the other hand,
since
$(0,s')$ is empty,
we have
$\P'=\Q'=\emptyset$.
For $s-s' < u \le v+1$,
we have $u-s+s' > 0$,
and hence $i(\P',u-s+s',s') = -1$.
Also,
for $w' < u \le v+1$,
we have $u-w' > 0$.
Since $w = \alpha(\D,s',s) \le \alpha(\F\cup\D,s',s) = w'$,
it follows that $u-w > 0$,
and hence $i(\Q',u-w,s') = -1$.
Therefore,
$$
\begin{array}{ll}
i(\P,u,s) = -1 = i(\P',u-s+s',s')
&\textrm{for } s-s' < u \le v+1,\\
i(\Q,u,s) = -1 = i(\Q',u-w,s')
&\textrm{for } w' < u \le v+1.
\end{array}
$$

Case 2: $s'>0$.
Since $(s',s)$ is the rightmost maximal vertebrate range of
$(\P,\Q) \succ \J(0,s)$
with $\U(s',s) \subseteq \C \subseteq \P$,
we have $(s'-1,s')\in\Q$ and $(s',s'+1)\in\P$.
For $0 \le i < s' < s \le m$,
it then follows by Lemma~\ref{lem:cut} that
\begin{align*}
\alpha(\P,i,s)
&= \alpha(\P(0,s'),i,s') + \alpha(\P(s',s),s',s)
= \alpha(\P',i,s') + \alpha(\C,s',s),\\
\alpha(\Q,i,s)
&= \alpha(\Q(0,s'),i,s') + \alpha(\Q(s',s),s',s)
= \alpha(\Q',i,s') + \alpha(\D,s',s).
\end{align*}
Since
$\alpha(\C,s',s) = s-s'$
and
$\alpha(\D,s',s) = w$,
we have
\begin{align*}
\alpha(\P,i,s) &= \alpha(\P',i,s') + s-s',\\
\alpha(\Q,i,s) &= \alpha(\Q',i,s') + w,
\end{align*}
which implies that
$$
\begin{array}{ll}
i(\P,u,s) = i(\P',u-s+s',s')
&\textrm{for } s-s' < u \le v+1,\\
i(\Q,u,s) = i(\Q',u-w,s')
&\textrm{for } w < u \le v+1.
\end{array}
$$
In particular, since $w \le w'$, we have
$$
\begin{array}{ll}
i(\P,u,s) = i(\P',u-s+s',s')
&\textrm{for } s-s' < u \le v+1,\\
i(\Q,u,s) = i(\Q',u-w,s')
&\textrm{for } w' < u \le v+1.
\end{array}
$$

Since $(\boldsymbol{p},\boldsymbol{q})$ encodes $(\P,\Q)$
and $(\boldsymbol{p'},\boldsymbol{q'})$ encodes $(\P',\Q')$,
it follows that for $1 \le u \le v+1$,
$$
p_u = i(\P,u,s),
\quad
q_u = i(\Q,u,s),
\quad
p'_u = i(\P',u,s),
\quad
q'_u = i(\Q',u,s).
$$
In summary, we have proved that
\begin{align*}
p_u &= \left\{
\begin{array}{ll}
s-u
& \quad\textrm{for } 1 \le u \le \min\{s-s',v+1\},\\
p'_{u-s+s'}
& \quad\textrm{for } s-s' < u \le v+1,
\end{array}
\right.
\\
q_u &= \left\{
\begin{array}{ll}
i(\F\cup\D,u,s)
& \quad\textrm{for } 1 \le u \le \min\{w',v+1\},\\
q'_{u-w}
& \quad\textrm{for } w' < u \le v+1.
\end{array}
\right.
\end{align*}
Thus
$(\boldsymbol{p},\boldsymbol{q})$
extends
$(\boldsymbol{p'},\boldsymbol{q'})$
by $(\E,\F)$.
\end{proof}

\subsection{Dynamic programming}

We next show how to decide
whether $\J$ admits a good $2$-partition by dynamic programming.
By Lemma~\ref{lem:basic}, it suffices to check
whether $\J$ admits a good basic $2$-partition.

For $0 \le s \le m$,
for each pair of $s$-monotonic sequences $(\boldsymbol{p},\boldsymbol{q})$
and each $2$-partition $(\A,\B)$ of $\K_s$,
denote by
$Y(s, \boldsymbol{p}, \boldsymbol{q}, \A, \B)$
the predicate whether there is a good basic $2$-partition
$(\P,\Q)$ of $\J(0,s)$
such that
$(\boldsymbol{p},\boldsymbol{q})$ encodes $(\P,\Q)$,
$\P\mid\P\cup\A$,
$\Q\mid\Q\cup\B$,
and $(s-1,s)\in\P$ when $s > 0$.
Then,
since $\J(0,m) = \J$ and $\K_m = \emptyset$,
$\J$ admits a good basic $2$-partition if and only if
there exists a pair of $m$-monotonic sequences $(\boldsymbol{p},\boldsymbol{q})$
such that
$Y(m,\boldsymbol{p},\boldsymbol{q},\emptyset,\emptyset) = \mathrm{true}$.

We can compute
$Y(s, \boldsymbol{p}, \boldsymbol{q}, \A, \B)$
by dynamic programming.
For the base case when $s = 0$,
since $\J(0,0)$ and $\K_0$ are both empty,
we have
$Y(0,\boldsymbol{0},\boldsymbol{0},\emptyset,\emptyset) = \mathrm{true}$.
For $1 \le s \le m$,
we can compute
$Y(s,\boldsymbol{p},\boldsymbol{q},\A,\B)$ by the recurrence in the following lemma:

\begin{lemma}\label{lem:dp}
For $1 \le s \le m$,
for any pair of $s$-monotonic sequences $(\boldsymbol{p},\boldsymbol{q})$
and any $2$-partition $(\A,\B)$ of $\K_s$,
$Y(s,\boldsymbol{p},\boldsymbol{q},\A,\B) = \mathrm{true}$ if and only if
for some $s'$, $0 \le s' < s$,
there exist a pair of $s'$-monotonic sequences $(\boldsymbol{p'},\boldsymbol{q'})$
and a $2$-partition $(\A',\B')$ of $\K_{s'}$,
such that the following conditions are satisfied
with $(\C,\D) = (\J_\le(s',s),\J_>(s',s))$:
\begin{enumerate}\setlength\itemsep{0pt}

\item
$Y(s',\boldsymbol{q'},\boldsymbol{p'},\B',\A') = \mathrm{true}$.

\item
$(\A',\B') \simeq (\A,\B)$.

\item
$(\boldsymbol{p},\boldsymbol{q})$
extends
$(\boldsymbol{p'},\boldsymbol{q'})$
by $(\A'\setminus\A,\B'\setminus\B) \succ \K_{s'}\setminus\K_s$.

\item
$\C \mid \A'\cup\C\cup\A$
and
$\D \mid \B'\cup\D\cup\B$.

\item
For each interval $(a,b) \in \K_{s'}\setminus\K_s$,
where $0 \le a < s' < b \le s$,
\begin{itemize}\setlength\itemsep{0pt}
\item
if $(a,b) \in \A'\setminus\A$,
then
$\alpha(\boldsymbol{p'},a,s') + \alpha(\C\cup(\A\setminus\A'),s',b) \le v$,
\item
if $(a,b) \in \B'\setminus\B$,
then
$\alpha(\boldsymbol{q'},a,s') + \alpha(\D\cup(\B\setminus\B'),s',b) \le v$.
\end{itemize}

\end{enumerate}
\end{lemma}

\begin{proof}
We first prove the direct implication.
Suppose that $Y(s,\boldsymbol{p},\boldsymbol{q},\A,\B) = \mathrm{true}$.
Then there exists
a good basic $2$-partition $(\P,\Q)$ of $\J(0,s)$
such that
$(\boldsymbol{p},\boldsymbol{q})$ encodes $(\P,\Q)$,
$\P \mid \P\cup\A$,
$\Q \mid \Q\cup\B$,
and $(s-1,s)\in\P$.
Since $\J(0,s)\cap\K_s = \emptyset$,
we have $(\P,\Q) \simeq (\A,\B)$,
and hence $(\P\cup\A,\Q\cup\B) \succ \J(0,s)\cup\K_s$.

Let $(s',s)$, $0 \le s' < s$, be the rightmost maximal vertebrate range for $(\P,\Q)$.
Since $(\P,\Q)$ is a good basic $2$-partition of $\J(0,s)$
with $(s-1,s)\in\P$,
and since $(\C,\D) = (\J_\le(s',s),\J_>(s',s)) \succ \J(s',s)$,
we must have $\C = \P(s',s)$ and $\D = \Q(s',s)$.
Let $\P' = \P(0,s')$ and $\Q' = \Q(0,s')$.
Then $(\P',\Q')$ is a good basic $2$-partition of $\J(0,s')$.

Note that $\K_{s'} \subseteq \J(0,s)\cup\K_s$.
Let $\A' = (\P\cup\A)\cap\K_{s'}$
and $\B' = (\Q\cup\B)\cap\K_{s'}$.
Then
$(\A',\B')$ is a $2$-partition of $\K_{s'}$
derived from the $2$-partition $(\P\cup\A,\Q\cup\B)$ of $\J(0,s)\cup\K_s$,
and is consistent with both $(\P,\Q)$ and $(\A,\B)$.
In particular, $(\A',\B') \simeq (\A,\B)$,
satisfying condition~2.
Moreover,
since both $(\P',\Q')$ and $(\C,\D)$ are derived from $(\P,\Q)$,
it follows that
$(\P,\Q)$, $(\P',\Q')$, $(\A,\B)$, $(\A',\B')$, and $(\C,\D)$ are pairwise consistent,
and so are all $2$-partitions derived from their combinations,
for example,
\begin{align*}
(\P'\cup\A',\Q'\cup\B') &\succ \J(0,s')\cup\K_{s'},\\
(\A'\setminus\A,\B'\setminus\B) &\succ \K_{s'}\setminus\K_s,\\
(\A'\cup\C\cup\A,\B'\cup\D\cup\B) &\succ \K_{s'}\cup\J(s',s)\cup\K_s,\\
(\C\cup(\A\setminus\A'),\D\cup(\B\setminus\B')) &\succ \J(s',s)\cup(\K_s\setminus\K_{s'}).
\end{align*}

For condition~1,
recall that
$(\P',\Q')$ is a good basic $2$-partition of $\J(0,s')$,
and is consistent with $(\A',\B') \succ \K_{s'}$.
Let $\boldsymbol{p'}$ and $\boldsymbol{q'}$ be the $s'$-monotonic sequences
encoding $\P'$ and $\Q'$, respectively.
Note that the following $2$-partitions are pairwise consistent:
\begin{gather*}
(\P',\Q') \succ \J(0,s'),\qquad
(\P,\Q) \succ \J(0,s)\\
(\P'\cup\A',\Q'\cup\B') \succ \J(0,s')\cup\K_{s'},\qquad
(\P\cup\A,\Q\cup\B) \succ \J(0,s)\cup\K_{s}.
\end{gather*}
Since $\J(0,s') \subseteq \J(0,s)$
and $\J(0,s')\cup\K_{s'} \subseteq \J(0,s)\cup\K_s$,
we must have
$$
\P' \subseteq \P,\qquad
\Q' \subseteq \Q,\qquad
\P'\cup\A' \subseteq \P\cup\A,\qquad
\Q'\cup\B' \subseteq \Q\cup\B.
$$
Since
$\P \mid \P\cup\A$
and
$\Q \mid \Q\cup\B$,
it follows that
$\P' \mid \P'\cup\A'$
and
$\Q' \mid \Q'\cup\B'$.
Also, by our choice of $s'$, we have either $s' = 0$ or $(s'-1,s')\in\Q'$.
Then $Y(s',\boldsymbol{q'},\boldsymbol{p'},\B',\A') = \mathrm{true}$.

Condition~3 is satisfied
by Lemma~\ref{lem:ef} with $(\E,\F) = (\A'\setminus\A,\B'\setminus\B) \succ \K_{s'}\setminus\K_s$.

For condition~4,
note that the intervals in $\J(0,s)\cup\K_s$ that intersect any interval in $\J(s',s)$
are all in $\K_{s'}\cup\J(s',s)\cup\K_s$.
Note that
the following $2$-partitions
are pairwise consistent:
\begin{gather*}
(\C,\D) \succ \J(s',s),\qquad
(\P,\Q) \succ \J(0,s),\\
(\A'\cup\C\cup\A,\B'\cup\D\cup\B) \succ \K_{s'}\cup\J(s',s)\cup\K_s,\qquad
(\P\cup\A,\Q\cup\B) \succ \J(0,s)\cup\K_s.
\end{gather*}
Since
$\J(s',s) \subseteq \J(0,s)$
and $\K_{s'}\cup\J(s',s)\cup\K_s \subseteq \J(0,s)\cup\K_s$,
we must have
$$
\C \subseteq \P,\qquad
\D \subseteq \Q,\qquad
\A'\cup\C\cup\A \subseteq \P\cup\A,\qquad
\B'\cup\D\cup\B \subseteq \Q\cup\B.
$$
Since
$\P \mid \P\cup\A$
and
$\Q \mid \Q\cup\B$,
it follows that
$\C \mid \A'\cup\C\cup\A$
and
$\D \mid \B'\cup\D\cup\B$.

Finally, proceed to condition~5.
Note that
$$
(\J(0,s)\cup\K_s)\setminus\K_{s'} = \J(0,s')\cup(\J(s',s)\cup(\K_s\setminus\K_{s'})),
$$
where $\J(0,s')$ includes the intervals to the left of $s'$,
and $\J(s',s)\cup(\K_s\setminus\K_{s'})$ includes the intervals to the right of $s'$.
Also note that
$(\P',\Q')\succ\J(0,s')$
and
$(\C\cup(\A\setminus\A'),\D\cup(\B\setminus\B'))\succ\J(s',s)\cup(\K_s\setminus\K_{s'})$.
By our choice of $s'$,
we have
$(s',s'+1) \in \P\cup\A$
and
$(s'-1,s') \in \Q\cup\B$
for the $2$-partition $(\P\cup\A,\Q\cup\B) \succ \J(0,s)\cup\K_s$.

Consider any interval $(a,b) \in \K_{s'}\setminus\K_s \subseteq \J(0,s)\cup\K_s$,
where $0 \le a < s' < b \le s$.
If $(a,b) \in \A'\setminus\A$,
then we have
$\alpha(\P\cup\A,a,b) = \alpha(\P',a,s') + \alpha(\C\cup(\A\setminus\A'),s',b)$
by Lemma~\ref{lem:cut}.
Since $\P \mid \P\cup\A$,
it follows that
$\alpha(\P\cup\A,a,b) \le v$,
and hence
$\alpha(\P',a,s') + \alpha(\C\cup(\A\setminus\A'),s',b) \le v$.
By Lemma~\ref{lem:alpha},
$\alpha(\P',a,s') \ge \alpha(\boldsymbol{p'},a,s')$.
Thus
$\alpha(\boldsymbol{p'},a,s') + \alpha(\C\cup(\A\setminus\A'),s',b) \le v$.
Similarly, if $(a,b) \in \B'\setminus\B$,
then
$\alpha(\boldsymbol{q'},a,s') + \alpha(\D\cup(\B\setminus\B'),s',b) \le v$.

Thus conditions 1 through 5 are all satisfied.

\bigskip
We next prove the reverse implication.
Suppose that for some $s'$, $0 \le s' < s$,
there exist a pair of $s'$-monotonic sequences $(\boldsymbol{p'},\boldsymbol{q'})$
and a $2$-partition $(\A',\B')$ of $\K_{s'}$,
such that conditions 1 through 5 are satisfied.

By condition~1, there exists a good basic $2$-partition
$(\P',\Q')$ of $\J(0,s')$,
such that
$(\boldsymbol{p'},\boldsymbol{q'})$ encodes $(\P',\Q')$,
$\P'\mid\P'\cup\A'$,
$\Q'\mid\Q'\cup\B'$,
and
$(s'-1,s')\in\Q'$ when $s' > 0$.

By condition~2,
$(\A',\B') \simeq (\A,\B)$.
Since $\J(0,s')\cap\K_{s'} = \emptyset$,
we have $(\P',\Q') \simeq (\A',\B')$,
and hence $(\P'\cup\A',\Q'\cup\B') \succ \J(0,s')\cup\K_{s'}$.
Also, since
$\J(0,s')\cap\K_s = \emptyset$,
we have
$(\P',\Q') \simeq (\A,\B)$.
Thus
$(\P',\Q')$,
$(\A',\B')$,
and
$(\A,\B)$
are pairwise consistent.
Moreover,
$(\C,\D)$ is consistent with $(\P',\Q')$, $(\A',\B')$, and $(\A,\B)$
because
$\J(s',s)$ is disjoint from $\J(0,s')$, $\K_{s'}$, and $\K_s$.
Thus
$(\P',\Q')$,
$(\A',\B')$,
$(\A,\B)$,
and
$(\C,\D)$
are pairwise consistent.

Let
$\P = \P'\cup(\A'\setminus\A)\cup\C$
and
$\Q = \Q'\cup(\B'\setminus\B)\cup\D$.
Then $(\P,\Q)$ is a $2$-partition of
$\J(0,s) = \J(0,s') \cup (\K_{s'}\setminus\K_s) \cup \J(s',s)$,
and is consistent with
$(\P',\Q')$,
$(\A',\B')$,
$(\A,\B)$,
and
$(\C,\D)$.
Moreover, $(\P,\Q)$ is basic,
with $(s',s)$ as the rightmost maximal vertebrate range
and with $(s-1,s)\in\P$,
because $(\P',\Q')$ is basic, $(\C,\D)$ is simple,
$\U(s',s)\subseteq\C\subseteq\P$,
and
$(s'-1,s')\in\Q'\subseteq\Q$ when $s' > 0$.

By Lemma~\ref{lem:ef},
$(\P,\Q)$ is encoded by
a pair of $s$-monotonic sequences
that extends
$(\boldsymbol{p'},\boldsymbol{q'})$ by
$(\E,\F) = (\A'\setminus\A,\B'\setminus\B) \succ \K_{s'}\setminus\K_s$.
Since 
any pair $(\boldsymbol{p},\boldsymbol{q})$
that
extends
$(\boldsymbol{p'},\boldsymbol{q'})$
by $(\E,\F)$
is uniquely determined by
$(\boldsymbol{p'},\boldsymbol{q'})$
and $(\E,\F)$,
$(\P,\Q)$ must be encoded by
the pair $(\boldsymbol{p},\boldsymbol{q})$
that extends
$(\boldsymbol{p'},\boldsymbol{q'})$ by $(\E,\F)$
in condition~3.

We next show that
$\P \mid \P\cup\A$
and
$\Q \mid \Q\cup\B$.
Note that
$\J(0,s) = \J(0,s')\cup(\K_{s'}\setminus\K_{s})\cup\J(s',s)$.
The three subfamilies
$\J(0,s')$,
$\K_{s'}\setminus\K_s$,
and
$\J(s',s)$
have neighborhoods
$\J(0,s')\cup\K_{s'}$,
$\J(0,s)\cup\K_s$,
and
$\K_{s'}\cup\J(s',s)\cup\K_s$,
respectively.

For $\J(0,s')$ and its neighborhood
$\J(0,s')\cup\K_{s'}$,
condition~1 guarantees that
$\P' \mid \P'\cup\A'$
and
$\Q' \mid \Q'\cup\B'$.

For
$\J(s',s)$
and its neighborhood
$\K_{s'}\cup\J(s',s)\cup\K_s$,
condition~4 guarantees that
$\C \mid \A'\cup\C\cup\A$
and
$\D \mid \B'\cup\D\cup\B$.

For
$\K_{s'}\setminus\K_s$
and its neighborhood
$\J(0,s)\cup\K_s$,
recall that by condition~1,
we have either $s' = 0$ or $(s'-1,s')\in\Q'$.
If $s'=0$, then $\K_{s'}\setminus\K_s$ is empty.
Otherwise, we have $(s'-1,s')\in\Q'\subseteq\Q\cup\B$
and $(s',s'+1)\in\C\subseteq\P\cup\A$
for the $2$-partition $(\P\cup\A,\Q\cup\B) \succ \J(0,s)\cup\K_s$.

Consider any interval $(a,b) \in \K_{s'}\setminus\K_s$,
where $0 \le a < s' < b \le s$.
If $(a,b) \in \A'\setminus\A$,
then by Lemma~\ref{lem:cut}
we have
$$
\alpha(\P\cup\A,a,b) = \alpha(\P',a,s') + \alpha(\C\cup(\A\setminus\A'),s',b).
$$
Condition~5 guarantees that
$$
\alpha(\boldsymbol{p'},a,s')
\le \alpha(\boldsymbol{p'},a,s') + \alpha(\C\cup(\A\setminus\A'),s',b) \le v.
$$
By Lemma~\ref{lem:alpha},
$\alpha(\P',a,s') = \alpha(\boldsymbol{p'},a,s')$
whenever $\alpha(\boldsymbol{p'},a,s') \le v$.
Thus
$$
\alpha(\P\cup\A,a,b) = \alpha(\P',a,s') + \alpha(\C\cup(\A\setminus\A'),s',b)
= \alpha(\boldsymbol{p'},a,s') + \alpha(\C\cup(\A\setminus\A'),s',b) \le v.
$$
Similarly, if $(a,b) \in \B'\setminus\B$,
then
$$
\alpha(\Q\cup\B,a,b) = \alpha(\Q',a,s') + \alpha(\D\cup(\B\setminus\B'),s',b)
= \alpha(\boldsymbol{q'},a,s') + \alpha(\D\cup(\B\setminus\B'),s',b) \le v.
$$
Thus $\A'\setminus\A \mid \P\cup\A$
and
$\B'\setminus\B \mid \Q\cup\B$.

We have shown that
$\P \mid \P\cup\A$
and
$\Q \mid \Q\cup\B$.
In particular,
$\P \mid \P$
and
$\Q \mid \Q$,
and hence $(\P,\Q)$ is good.
In summary, we have a good basic $2$-partition $(\P,\Q)$ of $\J(0,s)$
such that
$(\boldsymbol{p},\boldsymbol{q})$ encodes $(\P,\Q)$,
$\P\mid\P\cup\A$,
$\Q\mid\Q\cup\B$,
and $(s-1,s)\in\P$.
Thus $Y(s, \boldsymbol{p}, \boldsymbol{q}, \A, \B) = \mathrm{true}$,
as desired.
\end{proof}

\subsection{Time complexity}

We say that two intervals in $\J$ \emph{overlap significantly}
if the length of their intersection is at least $2v+1$.
Construct a graph $H$ with vertex set $\J$ such that two vertices in $H$
are connected by an edge if and only if the corresponding two intervals in $\J$
overlap significantly.
Compute connected components in $H$,
and partition $\J$ into \emph{groups} accordingly.
We say that a $2$-partition of $\J$ is
\emph{group-conforming}
if no two intervals from the same group are put in different parts.

\begin{lemma}\label{lem:group}
Any good $2$-partition of $\J$ must be group-conforming.
\end{lemma}

\begin{proof}
It suffices to show that any two intervals that overlap significantly
must be in the same part.
Let $A$ and $B$ be two intervals in $\J$ such that the length $\ell$
of their intersection $A \cap B$ is at least $2v+1$.
Then $A \cap B$ contains $\ell \ge 2v+1$ disjoint unit intervals in $\U\subseteq\J$.
If $A$ and $B$ are in two different parts of a $2$-partition of $\J$,
then either $A$ or $B$ would contain at least $\lceil (2v+1)/2 \rceil = v+1$
disjoint unit intervals in the same part,
forming a star $K_{1,v+1}$.
\end{proof}

\begin{lemma}\label{lem:pierce}
Any integer point of the line is contained in intervals from at most
$2v^2 + v$ groups.
\end{lemma}

\begin{proof}
Fix any point $x$.
For each interval of length at least $2v+1$ that contains $x$,
let its \emph{core} be any subinterval of length exactly $2v+1$ that still contains $x$.
If the cores of two intervals coincide,
then they overlap significantly and hence belong to the same group.
Thus the number of groups with at least one interval of length at least $2v+1$
containing $x$
is at most the number of distinct intervals of length exactly $2v+1$ containing $x$,
which is $2v$.

On the other hand, for each $\ell$, $1 \le \ell \le 2v$,
there are $\ell-1$ distinct intervals of length $\ell$ containing $x$.
In total,
there are at most $(2v)(2v-1)/2 = 2v^2 - v$ such intervals,
and hence there are at most $2v^2 - v$ distinct groups including such intervals in $\J$.

In summary,
the point $x$ is contained in intervals from at most
$2v + 2v^2-v = 2v^2 + v$ groups.
\end{proof}

Note that the number $m$ of maximal cliques is at most the number $n$ of vertices
in vertebrate interval graphs.
For each $s$, $1 \le s \le m \le n$,
there are $n^{O(v)}$
pairs of $s$-monotonic sequences $(\boldsymbol{p},\boldsymbol{q})$.
By Lemma~\ref{lem:pierce},
each $\K_s$ includes intervals from $O(v^2)$ groups,
and hence has $2^{O(v^2)}$ group-conforming $2$-partitions $(\A,\B)$.
For two tuples
$(s,\boldsymbol{p},\boldsymbol{q},\A,\B)$
and $(s',\boldsymbol{p'},\boldsymbol{q'},\A',\B')$,
all conditions in Lemma~\ref{lem:dp} can be checked in $n^{O(1)}$ time.
Thus the running time of our dynamic programming algorithm
is $2^{O(v^2)}n^{O(v)}$.
This completes the proof of Theorem~\ref{thm:2partition}.

\section{Open questions}

What is the minimum number $k$ such that
every interval graph with $n$ vertices admits a vertex partition
into $k$ induced subgraphs that are vertebrate interval graphs?
Is there a polynomial-time algorithm for partitioning an interval graph
into the minimum number of vertebrate interval graphs?

\end{document}